\documentclass[12pt,reqno]{amsart}
\usepackage{amssymb}
\usepackage{amscd}
\usepackage[all]{xy}
\usepackage{amsfonts}
\usepackage{mathrsfs}

\numberwithin{equation}{section} \theoremstyle{plain}
\newtheorem{theorem}[subsection]{Theorem}
\newtheorem{proposition}[subsection]{Proposition}
\newtheorem{lemma}[subsection]{Lemma}
\newtheorem{corollary}[subsection]{Corollary}

\theoremstyle{remark}
\newtheorem{remark}[subsection]{Remark}
\newtheorem{example}[subsection]{Example}

\renewcommand{\leq}{\leqslant}
\renewcommand{\geq}{\geqslant}
\newsavebox{\proofbox}
\savebox{\proofbox}{\begin{picture}(7,7)  \put(0,0){\framebox(7,7){}}\end{picture}}

\newcommand\Z{\mathbb{Z}}

\newcommand\C{\mathbb{C}}

\newcommand\G{\mathbf{G}}
\newcommand\A{\mathbb{A}}

\newcommand\Spec{\operatorname{Spec}}

\newcommand\K{\mathbb{K}}
\newcommand\SL{\operatorname{SL}}

\newcommand\GL{\operatorname{GL}}

\newcommand\Hom{\operatorname{Hom}}

\newcommand\F{\mathbb{F}}
\newcommand\Q{\mathbb{Q}}

\renewcommand\O{\mathcal{O}}
\renewcommand\P{\mathcal{P}}

\newcommand\eps{\varepsilon}

\subjclass[2000]{Primary 20E45; Secondary 11R04, 11R44, 20D06, 20G30, 20G40}
\begin{document}

\title{On conjugacy growth of linear groups}
\thanks{The authors are grateful for grants from the ERC and the NSF}

\author[Breuillard]{Emmanuel Breuillard}
\author[Cornulier]{Yves de Cornulier}
\address[E.B. and Y.C.]{Laboratoire de Math\'ematiques\\
B\^atiment 425, Universit\'e Paris-Sud 11\\
91405 Orsay\\FRANCE}
\email{emmanuel.breuillard@math.u-psud.fr}
\email{yves.cornulier@math.u-psud.fr}

\author[Lubotzky]{Alexander Lubotzky}
\author[Meiri]{Chen Meiri}
\address[A.L. and C.M.]{Einstein institute of mathematics\\
Hebrew University\\
Jerusalem 91904\\
ISRAEL}
\email{alexlub@math.huji.ac.il}
\email{chen.meiri@mail.huji.ac.il}

\begin{abstract}
We investigate the conjugacy growth of finitely generated linear groups. We show that finitely generated non-virtually-solvable subgroups of $\GL_d$ have uniform exponential conjugacy growth and in fact that the number of distinct polynomials arising as characteristic polynomials of the elements of the ball of radius $n$ for the word metric has exponential growth rate bounded away from $0$ in terms of the dimension $d$ only.
\end{abstract}

\date{June 22, 2011}
\maketitle


\section{introduction}

Let $\Gamma$ be a finitely generated group which is generated by a
finite set $\Sigma$. Let $B_\Sigma(n)=(\Sigma \cup \Sigma^{-1})^n$ be the ball of radius $n$ in the Cayley
graph $\textrm{Cay}(\Gamma,\Sigma)$ of $\Gamma$ with respect to $\Sigma$, i.e., the set of elements in $\Gamma$ that
can be written as product of at most $n$ elements of $\Sigma\cup\Sigma^{-1}$. We denote by $|\cdot|$ the cardinality of a finite set and define
\begin{equation}\label{defalpha}\alpha_\Sigma:=\lim_{n \rightarrow \infty} \frac{\log|B_\Sigma(n)|}{n}.\end{equation}
The group $\Gamma$ is said to have exponential word growth if $\alpha_\Sigma>0$ for some (hence every) $\Sigma$ and uniform exponential word growth if $\inf_{\Sigma}\alpha_\Sigma>0$ when $\Sigma$ ranges over finite generating subsets. It follows from the Tits alternative \cite{tits} and the Milnor-Wolf theorem \cite{milnor,wolf} that non-virtually-nilpotent linear groups have exponential growth. Uniform exponential growth of these groups was established by Eskin-Mozes-Oh \cite{EMO} in characteristic zero and by Breuillard-Gelander in arbitrary characteristic \cite{breuillard-gelander}.

A related question advertised by Guba and Sapir in \cite{guba-sapir} and also discussed
in the forthcoming book \cite{mann} consists in determining the
\emph{conjugacy growth} of a group $\Gamma$ generated by a finite set
$\Sigma$. Namely, we are interested in the asymptotics of
the number $c_\Sigma(n)$ of conjugacy classes in $\Gamma$ intersecting the word ball $B_\Sigma(n)$ of
radius $n$. This question can be seen as a combinatorial analogue to the problem of counting the number of closed geodesics in a closed Riemannian manifold according to length, a problem much studied in the literature (see \cite{guba-sapir} and the references therein).
 Denote
$$\gamma_\Sigma:=\liminf_{n \rightarrow \infty} \frac{\log c_\Sigma(n)}{n}$$ and
say that $\Gamma$ has exponential conjugacy growth if $\gamma_\Sigma>0$
and uniform exponential conjugacy growth if $\inf_\Sigma\gamma_\Sigma>0$.

Rivin \cite[Obs. 12.4, \S 13]{rivin} computed the asymptotics of $c_\Sigma(n)$ for free groups. Ivanov \cite[\S 41.5]{olshanskii} proved the existence of groups with exponential growth and finitely many conjugacy classes; Osin \cite{osin} improved the result to get only two conjugacy classes. The conjugacy growth can therefore be dramatically smaller than the word growth. Guba and Sapir gave many examples of groups with exponential conjugacy growth and asked about other families of groups. In this
paper we answer their question for linear groups.

\begin{theorem}\label{linear} Let $\Gamma$ be a linear group, i.e.\ isomorphic to a subgroup of $\GL_d(F)$ for some field $F$, and suppose that $\Gamma$ is not virtually nilpotent. Then $\Gamma$ has uniform exponential conjugacy growth.
\end{theorem}

The case of virtually solvable groups, linear or not, was treated in \cite{breuillard-cornulier} (and independently by M.\ Hull in \cite{hull} in the polycyclic case): such groups have uniform exponential conjugacy growth unless they are virtually nilpotent. So in this paper we focus on non-virtually-solvable linear groups. We actually consider the finer problem of counting, given a finitely generated subgroup $\Gamma$ in $\GL_d(F)$, the number of $\GL_d$-conjugacy classes in the balls of $\Gamma$, resulting in the following theorem, which immediately entails Theorem \ref{linear}.

\begin{theorem}\label{mainthm} For every integer $d$, there exists a constant $c(d)>0$ such that if $F$ is a field and $\Sigma$ a finite symmetric subset of $\GL_d(F)$ generating a non-virtually-solvable subgroup, then $$\liminf_{n \rightarrow \infty} \frac{1}{n} \log \chi_{\Sigma}(n)\geq c(d),$$ where $\chi_{\Sigma}(n)$ is the number of elements in $F[X]$ appearing as characteristic polynomials of elements of $\Sigma^n$.
\end{theorem}

Combining this with exponential conjugacy growth in the solvable case \cite{breuillard-cornulier} and some further simple remarks in the solvable case (Proposition \ref{growsol}), we get the following trichotomy
\begin{corollary}
Let $F$ be any field and let $\Gamma$ be a finitely generated subgroup of $\GL_d(F)$. Then exactly one of the following holds
\begin{enumerate}
\item $\Gamma$ is virtually nilpotent (so has polynomial growth);
\item $\Gamma$ is virtually solvable but not virtually nilpotent; it has exponential conjugacy growth, while $\chi_\Sigma(n)$ is bounded above by a polynomial whose degree depends only on $\Gamma$ and not on $\Sigma$;
\item $\Gamma$ is not virtually solvable and then $\chi_\Sigma(n)$ grows exponentially with a rate bounded below by a constant $\mu>0$ depending only on $d$.
\end{enumerate}
\end{corollary}
This is summarized in the following table.

\vspace{.1cm}
\makebox[\linewidth]{
\footnotesize
\begin{tabular}{|c|c|c|c|}
\hline
$\Gamma$ & growth  & conjugacy  & characteristic poly-\\
 & $|B_\Sigma(n)|$ & growth $b^c_\Sigma(n)$ & nomial growth $\chi_\Sigma(n)$ \\
\hline
v.~nilpotent & {\em polynomial} & {\em polynomial} & {\em polynomial}\\
\hline
v.~solvable not v.~nilpotent & {\em exponential} & {\em exponential} & {\em polynomial}\\
\hline
not v.~solvable & {\em exponential} & {\em exponential} & {\em exponential}\\
\hline
\end{tabular}
}
\vspace{0.2cm}

Note that we have claimed here a strong form of uniformity, in which the rate of exponential conjugacy growth $\gamma_\Sigma$ depends only on $d$ and not on the subgroup $\Gamma$ of $\GL_d$ nor the field $F$. We will make use here of the fact, proved by the first named author in \cite{breuillard1} building on the earlier works \cite{EMO,breuillard-gelander,breuillard0} that non-virtually-solvable linear groups in $\GL_d$ have a word growth rate bounded from below by a positive lower bound depending only on $d$ and not on the field of definition. This used as a key ingredient the main result of \cite{breuillard0} which solved a semisimple analogue of the (still open) Lehmer conjecture from diophantine geometry. The uniformity for the whole class of solvable non-virtually-nilpotent subgroups of $\GL_2(\mathbb{C})$, for which a positive answer would imply the validity of the classical Lehmer conjecture \cite{breuillard2}, is still an open question.

\vspace{0.2cm}

\noindent\emph{About the proof.} A standard specialization argument shows that it is enough to prove Theorem \ref{mainthm} in the case where $\K$ is a global field, i.e.\ a finite extension of $\Q$ or $\mathbb{F}_p(t)$. Besides, the proof essentially boils down to the case where the Zariski closure $\G$ of $\langle\Sigma\rangle$ in $\GL_d$ is semisimple.
Then using strong approximation (Weisfeiler \cite{weisfeiler}, Pink \cite{pink}) for Zariski-dense subgroups of simple algebraic groups, the more recent Product Theorem
of Pyber-Szab\'o and Breuillard-Green-Tao \cite{pyber-szabo,breuillard-green-tao} on the classification of approximate subgroups of simple algebraic groups over finite fields, and a pigeonhole
argument using classical results about the distribution of primes, we prove that for many prime ideals $\P$ of the ring of
integers $\O_\K$ whose norm $|\P|:=|\O_\K/\P|$ is exponential in $n$,
the reduction map $\G(\O_\K)\to\G(\O_\K/\P)$ is surjective when
restricted to $B_\Sigma(Cn)$, where $C$ is a constant depending on
$d$ only. At this point we use the fact that the number of distinct
characteristic polynomials of elements of $\G(\O_\K/\P)$ depends
polynomially on $|\P|$ and thus is exponential in $n$.

\vspace{0.2cm}

In fact our methods can yield variants of Theorem \ref{mainthm}, see Section~\ref{concluding}. For example, if the Zariski closure $\G$ of $\Gamma$ is a connected simple algebraic group and $P$ is an arbitrary non-constant polynomial function on $\G$, then $P$ achieves exponentially many values on $B_\Sigma(n)$, where the exponential rate of growth has a lower bound depending only on $d$. While it is possible to extend this latter result to the semisimple case using the same method, we do not include a proof in this paper for two reasons. Firstly some serious technicalities arise, in particular when applying strong approximation in positive characteristic due to the presence of Frobenius twists (see \cite{pink}). Secondly as shown to us by E.~Hrushovski (private communication), it is possible to give a completely different treatment of this theorem (including an extension of Theorem \ref{fiber} to semisimple groups). His approach avoids any appeal to strong approximation nor to the product theorem, but uses instead ideas from model theory and still reduces the counting problem to the ordinary word growth, hence to \cite{breuillard1}, as in Theorem \ref{cover} below.

\vspace{.3cm}

\noindent \emph{Outline of the paper.} The paper is organized as
follows. In Section \ref{sketch}, we give a sketch of proof in
the particular case of Zariski dense subgroups of $\SL_d(\mathbb{Z})$.
In
Section
\ref{pigeonhole} we give a
quantitative version
of the fact that reduction modulo a large prime is injective on
finite subsets. In Section \ref{sapp}, we derive a fast generation result for
the mod $p$ quotients of $\Gamma$ using the strong approximation
theorem and the results on approximate groups mentioned above, which
we recall in Section \ref{app}. In Section \ref{cov}, we show that the ball of
radius $n$ in $\Gamma$ cannot be covered by less than an exponential
number of proper hypersurfaces of $\G$ of bounded degree. There we
elaborate slightly more than what is needed for the immediate
application to Theorem \ref{mainthm}. Some of these further
applications are described in Section \ref{concluding}. The proof of \ref{mainthm}
is completed in Section~\ref{reduc}.

\vspace{.3cm}
\noindent \emph{Acknowledgements}. The first two authors would like to thank the Hebrew University for its hospitality during their visit in the autumn of 2010. We are grateful to M.~Sapir for bringing the problem of conjugacy growth of linear groups to our attention. We also thank A.~Chambert-Loir and E.~Hrushovski for useful conversations and A.~Mann for his comments on an early version of this paper.


\section{Sketch of proof: a particular case}\label{sketch}

We provide here a sketch of proof in the particular case of Zariski dense
subgroups of $\SL_m(\Z)$. It contains the highlights of the proof of the
general case, although the latter is technically more involved.

\begin{theorem}\label{partcase}For every $m\ge 2$, there exists a constant $c=c(m)>0$ such that for every symmetric set $\Sigma$ in $\SL_m(\Z)$ generating a
Zariski-dense subgroup of $\SL_m$, there exists $N=N(\Sigma,m)\in \mathbb{N}$
such that for every $n \ge N(\Sigma,m)$ the number of traces of
elements of the ball $B_\Sigma(n)$ is at least
$e^{cn}$.
\end{theorem}

The proof will use the following three theorems. We state each one here in the case of our specific situation. The first theorem
asserts that $\SL_m(\Z)$ and its Zariski-dense subgroups have uniform exponential growth in a uniform way:
\begin{theorem}[Eskin-Mozes-Oh {\cite{EMO}}]\label{uni-sl}For every $m\ge 2$, there exists $\alpha=\alpha(m)>0$ such that $|B_\Sigma(n)|\ge e^{\alpha
n}$ for every $n\in \mathbb{N}$ and every symmetric subset
$\Sigma$ in $\SL_m(\Z)$ generating a Zariski-dense subgroup.
\end{theorem}
 Although Eskin-Mozes-Oh only state their theorem in \cite{EMO} for a fixed subgroup of $\SL_m(\Z)$, their proof carries over without any changes to yield the above result uniformly over the Zariski-dense subgroups of $\SL_m(\Z)$. For the general case of our Theorem \ref{linear}, we will require the more general uniformity result established in \cite{breuillard1}, where it is shown that the rate of growth can be bounded below by a uniform constant independently of the ring of definition.

 The second is the recently established Product Theorem:
\begin{theorem}[Breuillard-Green-Tao \cite{breuillard-green-tao}, Pyber-Szab\'o \cite{pyber-szabo}]\label{spec-prod} For every $\delta>0$ there
exists  a number $N_\delta=N_\delta(m)>0$ such that for every prime number $p$
and every symmetric generating subset $A$ of $\SL_m(\Z/p\Z)$ of size at least
$p^\delta$ we have $A^{N_\delta}=\SL_m(\Z/p\Z)$.
\end{theorem}
The third is the Strong Approximation Theorem:
\begin{theorem}[Matthews-Vaserstein-Weisfeiler \cite{matthews-vaserstein-weisfeiler}]\label{mvw} If $\Gamma$ is a Zariski-dense subgroup of $\SL_m(\Z)$, then for all but finitely many primes $p$, we have $\pi_p(\Gamma)=\SL_m(\Z/p\Z)$, where $\pi_p$ is the reduction mod $p$ map.
\end{theorem}

Let $\Sigma$ be as in the statement of Theorem \ref{partcase} and define
$C:=\max_{s\in \Sigma}\|s\|$ where $\|T\|$ denotes the operator norm of a
matrix $T$. For every $n \in \mathbb{N}$ large enough, choose a symmetric subset $B_n$ of
$B_\Sigma(n)$ of size $e^{\alpha n}$ containing $\Sigma$, where $\alpha$ is given by Theorem \ref{uni-sl}.

We claim that there is $k_0 \in \mathbb{N}$ such that for $k \ge k_0$ there exists a
prime number $p$ in the interval $[e^{2\alpha k},e^{4\alpha k}]$
such that the restriction of the map $\pi_p:\SL_m(\Z) \rightarrow
\SL_m(\Z/p\Z)$ to $B_k$ is injective.

If $k$ is large enough, by the distribution of the prime numbers (e.g.\ Chebyshev's estimate, see Theorem \ref{cheby}), there exist at least $e^{3\alpha k}$ prime numbers in the interval
$[e^{2\alpha k},e^{4\alpha k}]$. For each $(g,h)\in B_k\times B_k$, each nonzero entry of the matrix $gh^{-1}-I_m$ is at most $C^{2k}$, so the number of its prime divisors greater than $e^{2\alpha k}$ is at most $\frac{\log(C^{2k})}{\log(e^{2\alpha k})}=\frac{\log C}{\alpha}$. We have $m^2$ entries for each $gh^{-1}$, and $\le e^{2\alpha k}$ possible pairs $(g,h)$, so the number of primes greater than $e^{2\alpha k}$ dividing at least one nonzero entry of $gh^{-1}-I_m$ for some $(g,h)\in B_k\times B_k$ is $\le \frac{\log C}{\alpha}m^2e^{2\alpha k}$, which is less than $e^{3\alpha k}$ for $k$ large enough. Thus, by the pigeon-hole principle, if $k$ is large enough, there exists a prime $p$ in $[e^{2\alpha k},e^{4\alpha k}]$ not dividing any nonzero coefficient of $gh^{-1}-I_m$ for any $(g,h)\in B_k\times B_k$. This means that $B_k$ maps injectively into $\SL_m(\Z/p\Z)$.

Let $N=N_{1/4}$ be the constant
in Theorem \ref{spec-prod}. Let $n\ge N(k_0+1)$ and $k=\lfloor n/N\rfloor$; by the above we can fix a prime $p$ so that $\pi_p|_{B_k}$ is injective.
By Theorem \ref{mvw}, $\pi_p(B_{k})$ generates $\SL_m(\Z/p\Z)$ as soon as $n$ is large enough. Moreover, we have $| \pi_p(B_{k})| \ge e^{\alpha k}\ge p^{1/4}$, so
Theorem \ref{spec-prod} implies that
$\pi_p(B_\Sigma(n))=\SL_m(\Z/p\Z)$. Since $m\ge 2$, every element of $\Z/p\Z$ is a
trace of some matrix in $\SL_m(\Z/p\Z)$; accordingly the number of traces of
elements which belong to $B_\Sigma(n)$ is at least $p \ge e^{2\alpha
k}\ge e^{\alpha n/N}$, and this yields the assertion of Theorem \ref{partcase}.


\section{Preliminaries and notation}\label{preliminaries}

\subsection{Functions}

For real-valued functions, we write $f(x)\preceq g(x)$ or
$f(x)=O(g(x))$ if for some constant $C>0$ we have $f(x)\le Cg(x)$
for all large $x$. If $f(x)\preceq g(x)\preceq f(x)$ we write
$f(x)\approx g(x)$.

\subsection{Global fields}

The letter $\K$ will always denote a global field, that is
either a number field, i.e.\ a finite extension of $\Q$, or a
function field, i.e.\ a finitely generated field of transcendence
degree one over a finite field. We denote by $\overline{\K}$ an
algebraic closure of $\K$. A place on $\K$ is the norm induced by the embedding of $\K$ into a nondiscrete locally compact field. We identify equivalent places, i.e.\ places inducing the same topology on $\K$.

Let $S$ be a nonempty finite set of places on $\K$ including all
Archimedean ones. The ring of $S$-integers of $\K$, defined as
$\O_\K(S)=\{a\in\K:\;\forall v\notin S,\; v(a)\le 1\}$ is a subring
of $\K$ whose field of fractions is $\K$. Moreover, $\O_\K(S)$ is a
finitely generated Dedekind domain. Let $\Spec(\O_\K(S))$ be the set
of its prime ideals, consisting of $\{0\}$ along with infinitely
many maximal ideals of finite index. If $\P\in\Spec(\O_\K(S))$ is
nonzero, the size of the residue field $|\P|=|\O_\K(S)/\P|$ is
called the norm of $\P$. If $\P=\{0\}$ we set $|\P|=0$.

Let $V_\K$ be the set of all places of $\K$. For every $v \in V_\K$, let $\K_v$ be the completion of $\K$ with respect to $v$. Let $\A=\prod_{v \in S} \K_v$. If $v$ is a place associated to a prime ideal $\P$ of $\O_\K$, we may choose for $|\cdot|_v$ the absolute value $|x|_v=q^{-\nu_{\P}(x)}$, where $q$ is the size of the residue field $\O_\K/\P$ and $\nu_{\P}(x)$ the $\P$-valuation of $x$, so that the product formula holds for all $x \in \K^{\times}$, $\prod_{v \in V_\K}|x|_v=1$.
Let $\|\cdot\|_v$ be the standard norm on $\K_v^d$ relative to
$|\cdot|_v$, i.e.\ the Euclidean (or Hermitian) norm if $v$ is Archimedean and the supremum norm if $v$ is non-Archimedean (i.e.\ $|x|_v=\max|x_i|_v$). We also denote by $\|\cdot\|_v$ the associated operator norm on $\GL_d(\K_v)$ and we let $\|(g_v)_v\|=\max\|g_v\|_v$ for all $g=(g_v)_v \in \GL_d(\A)$ and $|a|=\max|a_v|_v$ for all $a=(a_v)_v \in \A$.

If $S\subset S'$ then $\O_\K(S')$ is a localization of $\O_\K(S)$ and $\Spec(\O_\K(S'))$ is the complement of a finite subset of $\Spec(\O_\K(S))$. Moreover, for any $\P\in\Spec(\O_\K(S'))$ there is a canonical field isomorphism between residual fields $\O_\K(S)/(\P\cap\O_\K(S))\to\O_\K(S')/\P$. We thus write $\K_\P=\O_\K(S)/\P$. In particular, apart from finitely many primes, $\Spec(\O_\K(S))$ and its norm function do not depend on $S$ and we thus speak of ``primes of $\K$" whenever this finite indeterminacy is irrelevant. For instance we will use

\begin{theorem}[Chebyshev, Landau {\cite[Theorem 5.12]{rosen}, \cite[Theorem 7]{hardy-wright}}]\label{cheb}
Let $\K$ be a global field. Let $\pi(x)$ be the number of primes of $\K$ of norm $\le x$. Then
$$\pi(x)\approx\frac{x}{\log(x)}\quad (x\to+\infty).$$\label{cheby}
\end{theorem}

This is a weak form of the prime number theorem, which asserts that $\frac{\pi(x)}{x/\log(x)}$ actually tends to~1. Theorem \ref{cheby} has the following consequence.

\begin{lemma}\label{fgen1}
Let $\K_0\subset\K$ be an extension of global fields. Then the
number of primes $\P$ of $\K$ of norm $\le x$ such that $f_\P>1$,
where $f_\P=[\O_\K/\P:\O_{\K_0}/(\P\cap\O_{K_0})]$ is the residual
degree, is $o(\sqrt{x})$. In particular, the number of primes of
$\K$ of norm $\le x$ and with $f_\P=1$ is $\approx x/\log(x)$.
\end{lemma}
\begin{proof}

Let $\mathfrak{p}=\P \cap \O_{\K_0}$ be the prime below $\P$. We
have $|\P| = |\mathfrak{p}|^{f_\P}$. If $f_\P \geq 2$ and $|\P|
\leq x$ it follows that $|\mathfrak{p}| \leq \sqrt{x}$. Since there
are at most $[\K:\K_0]$ primes $\P$ above any given prime of
$\O_{\K_{0}}$, by Chebyshev's theorem there are at most
$O(\sqrt{x}/\log(x))$ primes $\P$ of $\K$ with $f_\P \geq 2$ and $|\P| \leq x$. We are done.
\end{proof}

\subsection{Reduction modulo a prime}\label{reductionprime}

Let $\K$ be a global field. Let $\G$ be a linear algebraic group defined over $\K$. We want to define ``reduction modulo~$\P$" for $\G$.
Let $A$ be a finitely generated subdomain with $\K$ as field of fractions.
We can write the ring of functions as
$\K[\G]=M\otimes_{A_b}\K$, where $A_b$ is a suitable localization of
$A$ and $M\subset\K[\G]$ a Hopf algebra over $A_b$. This choice
being made, we write $A_b[\G]$ instead of $M$. Thus for any $A_b$-algebra $B$ we
can define functorially $\G(B)=\Hom(A_b[\G],B)$ which is naturally a
group. In particular $\G(A/\P)$ is well-defined for every $\P \in
\Spec(A_b)$, and the reduction mod~$\P$ map is the group
homomorphism $\G(A_b)\to\G(A_b/\P)$.

This depends on the choice of the Hopf algebra structure $M$ over
$A_b$; if two different choices $M_i$ over $A_{b_i}$ are made giving
rise to forms $\G_i$ of $\G$ over $A_{b_i}$, the identity induces an
isomorphism $M_1\otimes_{A_{b_1}}\K\simeq M_2\otimes_{A_{b_2}}\K$;
such an isomorphism is actually defined over a suitable common
localization $A_b$, and in particular, restricted to the class of
$A_b$-algebras, the functors $B\mapsto\Hom(M_i,B)$ are equivalent
for $i=1,2$. Given two fixed choices $M_i$ over $A_{b_i}$, $i=1,2$, the group scheme structures will coincide for all but finitely many $\P$'s.
Similarly, if $\Gamma$ is a finitely
generated subgroup of $\G(\K)$, then for $\P$ large enough, we can talk about  the
homomorphism $\Gamma\to\G(\K_\P)$, where $\K_\P=\O_\K(S)/\P$.

Moreover, $A_b[\G]\otimes_{A_b} \overline{\K}$ is a reduced ring and is
a domain if $\G$ is connected. This continues to hold modulo $\P$ for
$\P$ large enough, namely $\G$ is reduced over $\K_\P$, and is connected
if $\G$ is connected. Indeed, since $A_b[\G]$ is a flat $A_b$-module (if
we suppose as we may that $A_b$ is Dedekind, then flat means
torsion-free), ``geometrically reduced" and ``geometrically integral"
are open properties on $\Spec(A_b)$ \cite[12.1.1]{grothendieck}.

Finer arguments of the same flavour show that if $\G$ is semisimple and simply connected, then this still holds over $\K_\P$ for large $\P$.

\section{Reduction and pigeonholing on prime ideals}\label{pigeonhole}

In this section, we describe a pigeonhole argument (Corollary
\ref{pigeon2} below). In combination with Chebyshev's weak version
of the prime number theorem for global fields, this will yield many
good prime ideals modulo which the ``ball" $\Sigma^n$ will be
preserved.

As above $\K$ denotes a global field and $S$ a finite set of places including all Archimedean ones. Let $\{B_n\}_n$ be a family of finite subsets of $\GL_d(\O_\K(S))$ such that:

 \begin{itemize}
 \item $B_n\subset W^n(=W \cdot...\cdot W)$ for some finite subset $W$ of $\GL_d(\O_\K(S))$;
 \item $|B_n| \geq e^{\alpha n}$ for some fixed $\alpha >0$.
 \end{itemize}

The reader interested in a proof of Theorem \ref{mainthm} under the assumption that the Zariski closure of $\langle\Sigma\rangle$ is connected semisimple, can always suppose, in the forthcoming results, that $B_n=\Sigma^n$. This is, in particular, enough in order to obtain Theorem \ref{linear} (that is exponential conjugacy growth without the uniformity in the field claimed in Theorem \ref{mainthm}), because every non-virtually solvable linear group has a finite index subgroup with a quotient isomorphic to a Zariski-dense subgroup of a simple algebraic group.

Given a prime ideal $\P$ not in $S$, let $\pi_{\P}$ be the reduction mod $\P$ map from $\GL_d(\O_\K(S))$ to $\GL_d(\F_q)$, where $\F_q=\O_\K/\P$.

\begin{proposition}
Suppose  that $B_n \subset W^n$ are sets as above. There exists a constant $C=C(W,S)>0$ such that for all $n$, all $\gamma\in B_n^{-1}B_n$ and $\rho>1$ we have
$$\kappa_{\rho^n}(\gamma)\le\frac{C}{\log(\rho)},$$
where $\kappa_{\rho^n}(\gamma)$ is the number of primes $\P$ with $|\P| \geq \rho^{n}$ such that $\pi_{\P}(\gamma)=1$.
\end{proposition}

\begin{proof} We make use of the following easy consequence of the product formula: if $\P$ is a prime ideal in $\O_\K(S)$, then $|x|^{|S|} \geq |\P|$ for any $x \in \P\setminus \{0\}$. Similarly, if $g \in \GL_d(\O_{\K}(S))$, $g \neq 1$, and $g-1 \in M_d(\P_i)$ for $k$ distinct primes ideals $\P_1$,...,$\P_k$ not in $S$, then $\|g-1\|^{|S|}\geq |\P_1|...|\P_k|$.
So if $\pi_{\P_i}(\gamma)=1$ for each $\P_1$,...,$\P_k$, then
$\|\gamma-1\|^{|S|}\geq \rho^{nk}$. But $\|\gamma-1\| \leq 1+M^{2n} \leq M^{3n}$, where $M:=\max\{\|g\|,g \in W\}$. Hence the result.
\end{proof}

We then derive:

\begin{corollary}\label{pigeon2}\label{pigeon} With probability tending to $1$ as $n$ tends to infinity, a prime $\P$ of $\K$ whose norm $|\P|$ lies in the interval $[e^{3\alpha n},e^{4\alpha n}]$ must satisfy $|\pi_{\P}(B_n)| \geq |\P|^{\frac{1}{4}}$.
\end{corollary}

\begin{proof} Let $P_n$ be a subset of $B_n$ of size $e^{\alpha n}$. If $\pi_{\P}$ is not injective on $P_n$, then there must exist $\gamma \in P_n^{-1}P_n$ such that $\pi_{\P}(\gamma)=1$ while $\gamma \neq 1$. However by the last proposition, there are at most $\kappa:=C/3\alpha$ such prime $\P$ with norm $|\P| \geq e^{3\alpha n}$. Hence there are at most $\kappa|P_n|^2=O(e^{2\alpha n})$ possibilities for such a prime. However, by Chebyshev's theorem (Theorem \ref{cheb} above), there are $\approx e^{4 \alpha n} / n$ primes with norm in $[e^{3\alpha n},e^{4\alpha n}]$. Hence for most such primes $\pi_{\P}$ is injective on $P_n$, and thus $|\pi_{\P}(B_n)| \geq e^{\alpha n} \geq |\P|^{\frac{1}{4}}$.
\end{proof}

\section{Approximate subgroups and fast generation in semisimple algebraic groups}\label{app}

One of the key ingredients in the proof of our main theorem, is the following recent result regarding approximate subgroups of simple algebraic groups over finite fields.

Let $\G\subset\GL_d$ be an algebraic group defined over an algebraically closed field $k$. We will say that a closed algebraic subvariety $\mathcal{V}$ of $\G$ has bounded \emph{complexity} (say bounded by $M\geq 1$) if it is defined as the set of zeros of at most $M$ polynomial maps on $\G$ of degree at most $M$. We will also say that a subset of $\G$ is $M$-sufficiently Zariski dense if it is not contained in a proper closed algebraic subvariety of $\G$ of complexity at most $M$. For more details about this definition, we refer the reader to \cite{breuillard-green-tao} especially Section 3 and Appendix A therein.

The following was obtained in \cite{breuillard-green-tao}.

\begin{theorem}[Product Theorem]\label{approximate}
Let $\G$ be a (connected) almost simple linear algebraic group of
dimension $d$ defined over an algebraically closed field $k$. There
exist constants $\eps,C>0$, depending only on $d$ and not on $k$,
such that the following holds. Let $A$ be a finite subset of
$\G(k)$, then
\begin{itemize}
\item either $\langle A \rangle$ is not $C$-sufficiently Zariski-dense in $\G$, that is $A$ is contained in a proper algebraic subgroup of $\G$ of complexity at most $C$.
\item or $|AAA| \geq \min\{|\langle A \rangle|, |A|^{1+\eps}\}$.
\end{itemize}
\end{theorem}

The above was obtained independently by Pyber and Szab\'o (\cite{pyber-szabo}) in the case when $k=\overline{\F_p}$
 and $A$ generates $\G(\F_q)$, which is the hardest case and the only one we will use in this paper.

As a direct consequence, we get:

\begin{corollary}\label{fastgeneration} Let $\mathbf{H}$ be a simple algebraic group defined over a finite field $\F_q$, of dimension at most $d$. Let $\beta>0$. Then there is $D=D(\beta,d)>0$ such that the following holds: if $A$ is a finite generating subset of $\mathbf{H}(\F_q)$ such that $|A|\geq q^{\beta}$, then $A^{D}=\mathbf{H}(\F_q)$.
\end{corollary}

\section{Strong approximation}\label{sapp}
To apply Corollary \ref{fastgeneration}, we need to know that $\Gamma$ maps onto many mod~$\P$ quotients. This is a consequence of the so-called ``strong approximation", a result due to Weisfeiler \cite{weisfeiler}, except some tricky cases due to the existence of ``non-standard isogenies" in characteristic two or three, and the general result is due to Pink \cite{pink}. We have:

\begin{theorem}\label{sa}
Let $\K$ be a global field of characteristic $p$ (possibly $p=0$) and $\G\subset\GL_d$ be a simply connected absolutely simple $\K$-subgroup. Let $\Gamma$ be a finitely generated Zariski dense subgroup of $\G$ contained in $\G(\K)$. Then with probability tending to one when $x\to\infty$, if $\P$ is a prime of $\K$ with norm $\le x$, then $\pi_\P(\Gamma)=\G(\K_\P)$.
\end{theorem}

\begin{proof}
By Weisfeiler's theorem \cite[Theorem~1.1]{weisfeiler} (or Pink's version \cite{pink} in case of characteristic $2$ and $3$) there exists a finitely generated subfield $\K_0$ of $\K$ (namely the subfield generated by the traces of $\textnormal{Ad}(\Gamma)$ in characteristic 0) and a $\K_0$-structure on $\G$ such that $\Gamma\subset\G(\K_0)$ and for all $\mathfrak{p}\in\textnormal{Spec}(\O_{\K_0})$ large enough we have $\pi_\mathfrak{p}(\Gamma)=\G((\K_0)_{\mathfrak{p}})$.
 Let $\P$ be a prime of $\O_\K$ of norm $\le x$, with residual degree $f_\P=[\O_\K/\P:\O_{\K_0}/\mathfrak{p}]$, where $\mathfrak{p}=\P\cap\O_{\K_0}$. We can suppose that $f_\P=1$, since this holds with probability tending to one by Lemma \ref{fgen1}. Hence
$$\pi_\P(\Gamma)\supseteq \pi_{\mathfrak{p}}(\Gamma)=\G(\O_{\K_0}/\mathfrak{p})=\G(\O_\K/\P).$$
\end{proof}

Combining Theorem \ref{sa}, Corollary \ref{fastgeneration}, and Corollary \ref{pigeon2}, we obtain

\begin{corollary}\label{ballsurjects}
For every $d$ and $\alpha>0$ there exists $D=D(d,\alpha)$ such that the
following holds. Let $\K$ be a global field and
$\mathbf{H}$ be a  simply connected absolutely simple $\K$-group. Let
$\Gamma$ be a finitely generated Zariski-dense subgroup of
$\mathbf{H}(\K)$ and $W \subset \GL_d(\K)$ a finite subset.
Let $(B_n)_n$  be a family of subsets of $\mathbf{H}(\K)$ such that
\begin{itemize}
\item $B_n \subset W^n$ for every $n \geq 1$,
\item $\Gamma \subset \langle B_n \rangle$ for all $n$ large enough,
\item $|B_n|\geq e^{\alpha n}$.
\end{itemize}
  Then, with probability
tending to one as $n\to\infty$, if $\P$ is a prime of $\K$ of norm in
$[e^{3\alpha n},e^{4\alpha n}]$, we have
$$\pi_\P(B_{n}^D)=\mathbf{H}(\K_\P).$$
\end{corollary}
\begin{proof}
By Corollary \ref{pigeon2}, with probability tending to one as $n$ tends to $+\infty$, a prime $\P$ with norm $|\P| \in [e^{3\alpha n},e^{4 \alpha n}]$ satisfies $|\pi_\P(B_n)|\ge |\P|^{\frac{1}{4}}$. By Corollary \ref{fastgeneration}, there exists $D>0$ depending only on $d$ such that, provided $\pi_\P(\Gamma)=\mathbf{H}(\K_\P)$ for all $i$, we have $\pi_\P(B_{n}^D)=\mathbf{H}(\K_\P)$. Finally, the condition $\pi_\P(\Gamma)=\mathbf{H}(\K_\P)$ holds with probability tending to one by Theorem \ref{sa}.
\end{proof}

\section{Covering balls by subvarieties}\label{cov}

The following theorem indicates that in a simple algebraic group, large balls cannot be covered by a small number of subvarieties of bounded complexity. Let $\G$ be a connected simple algebraic group defined over a global field $\K$ with $d=\dim \G$. We fix a linear embedding $\G \leq \GL_d$.  Suppose that $\Gamma$ is a Zariski dense subgroup of $\G(\K)$. Let $W \subset \GL_d(\K)$ be a finite subset. Now let $(B_n)$ be a family of finite sets of $\G(\K)$ such that

\begin{itemize}
\item $B_n \subset W^n$ for every $n \geq 1$,
\item $\Gamma \subset \langle B_n \rangle$ for all $n$ large enough,
\item $|B_n| \geq e^{\alpha n}$ for some fixed $\alpha>0$.
\end{itemize}

\begin{theorem} \label{cover}Given $B_n$ and $\alpha>0$ as above and $M>0$ there exist $D=D(d,\alpha)\geq 1$ (independent of $M$) and $n_0=n_0(d,\alpha,M) \geq 1$, such that the following holds. Let $\Theta_n$ be the smallest $k\geq 1$ such that there are proper subvarieties $\mathcal{V}_1,...,\mathcal{V}_k$ of $\G$ with complexity bounded by $M$ such that $$B_n^D\subset \bigcup_{1 \le i \le k} \mathcal{V}_i.$$ Then $\Theta_n \geq  e^{\alpha n}$ for every  $n \geq n_0$.
\end{theorem}

We will use the following estimate on the number of points on a variety over a finite field.

\begin{proposition}\label{langweilremix}
Let $d,m$ be positive integers. There exists a constant $c=c(d,m)$ such that for every finite field $\mathbb{F}_q$ and every closed $r$-dimensional subvariety $X$ of the $d$-dimensional affine space over $\mathbb{F}_q$ of complexity $\le m$ we have
$$\# X(\mathbb{F}_q)\le cq^r.$$
\end{proposition}

This is probably well known to experts (modulo the definition of complexity), but in a lack of reference we provide a proof based on the Lang-Weil estimates, although they are probably also not needed for this upper bound.

\begin{proof}
A much more precise asymptotic behavior with upper and lower bounds
is given by the Lang-Weil theorem \cite{lang-weil}, but it requires the assumption that the
variety is absolutely irreducible. As we will see below, there is no
asymptotic lower bound by $q^r$ in case the variety is irreducible
but not absolutely irreducible.

Let us check however that the theorem follows from the original statement in \cite{lang-weil}. We argue by induction on the integer $r\in [0,d]$. Let us suppose that the theorem is proved for all $r'<r$ and let $X$ have dimension $r$. First, because of the bound on the complexity, we have a bound on the number of irreducible components \cite[Lemma A.4]{breuillard-green-tao}, and therefore it is enough to prove the theorem when $X$ is irreducible over $\mathbb{F}_q$ and $r$-dimensional.
\begin{itemize}
\item Suppose that $X$ is absolutely irreducible. Then the Lang-Weil Theorem (as stated in \cite{lang-weil}) directly provides the desired upper bound.
\item Suppose that $X$ is not absolutely irreducible. Let $X_1,\dots,X_k$ be the irreducible components of $X$. By \cite[Lemma A.4]{breuillard-green-tao}
     the integer $k$ can be bounded in terms of $d,m$. The components $X_i$ are defined over some finite extension of $\mathbb{F}_q$.
     This is a Galois extension, and $X$ is irreducible over $\mathbb{F}_q$, so the action of the Galois group on these components is
     transitive. Moreover, $X(\mathbb{F}_q)$ is contained in $Y=\bigcap X_i$.
     By assumption, $k\ge 2$, so $Y=\bigcap X_i$ has dimension $<r$ and is defined over $\mathbb{F}_q$ and has complexity bounded by some constant depending only on $m$ and $k$, hence of $d$ and $m$. So by induction we get $\# Y(\mathbb{F}_q)\le c'q^{r-1}$ for some constant $c'=c(m,d)$ and $$\# X(\mathbb{F}_q)\le\# Y(\mathbb{F}_q)\le c'q^{r-1}\le c'q^r.$$
\end{itemize}
Note that the induction has only $d$ steps, hence the constant $c$ eventually remains controlled by $(d,m)$.\end{proof}

\begin{proof}[Proof of Theorem \ref{cover}]

To apply Corollary \ref{ballsurjects}, we need to assume that $\G$ is simply connected. So first assume that the theorem is proved when $\G$ is simply connected and let us prove it in general. Let $\kappa:\tilde{\G}\to\G$ be the simply connected covering of $\G$; it is defined over $\K$; its kernel has cardinality bounded by some number only depending on $d$, and it has bounded degree.
Now $\kappa^{-1}(B_n)$ is also a family of generating subsets of $\kappa^{-1}(\Gamma)$ satisfying the required assumptions, and a covering of $B_n$ by $k$ proper subvarieties pulls pack to a covering of $\kappa^{-1}(B_n)$ by $k$ proper subvarieties. We can therefore assume that $\G$ is simply connected.

Let $S$ be some non-empty finite set of valuations on $\K$ including the Archimedean ones and such that $W \subset \GL_d(\O_\K(S))$. By Lemma \ref{principal} below, if we choose $S$ large enough, then $A=\O_\K(S)$ is a principal ideal ring.

Enlarging $S$ again if necessary, we can ensure that $A[\G]\otimes_{A}A/\P$ is a reduced ring for all primes $\P$ (a priori this holds for all but finitely many $\P$'s, see \S \ref{reductionprime}). We may also fix an $A$-structure on $\G$, i.e.\ we fix an isomorphism
$\K[\G]=A[\G]\otimes_{A}\K$, where $A[\G]\subset\K[\G]$ is a Hopf $A$-subalgebra.

Now suppose that $B_n^D\subset \bigcup_{i=1}^{k_n} X_i$ with $X_i$ of
complexity $\le M$. We can suppose without loss of generality that $X_i$
is given as a proper hypersurface $\{f_i=0\}$ in $\K[\G]$.

Now multiplying by a suitable nonzero element of $\K$ we can
even assume that $f_i\in A[\G]$. Moreover, if $f_i\in aA[\G]$
for some $a\in A-\{0\}$ then we can replace $f_i$ by $a^{-1}f_i$
without changing its set of zeros, so by noetherianity of the
domain $A[\G]$ we
can suppose that $f_i\notin aA[\G]$ for any $a\in A$ not
invertible in $A[\G]$ (or equivalently in $A$: because of the co-unity $A[\G]\to A$, if $a\in A$ is not invertible then it remains non-invertible in $A[\G]$).

We have the following claim: \emph{for every prime ideal $\P$ of $A$, $f_i$
defines
a proper
hypersurface $X_i^\P$ of $\K_\P[\G]$.}

Let us first finish the proof of Theorem \ref{cover}, granting the claim for a moment. For all $\P$'s we have $$\pi_\P(B_n^D)\subset
\bigcup_{i=1}^{k_n}X_i^\P(\K_\P).$$

By Corollary \ref{ballsurjects}, with probability tending to one as $n$ tends to $+\infty$, if
$\P$ is has norm in $[e^{3\alpha n},e^{4 \alpha n}]$, then $\pi_\P(B_n^D)=\G(\K_\P)$.
For such a prime, we get $$|\G(\K_\P)|\le k_n\sup_i
|X_i^\P(\K_\P)|.$$ If $d$ is the dimension of $\G$, and we use the shorthand $q:=|\P|$, then the
Lang-Weil upper bound in Theorem \ref{langweilremix} gives
$|X_i^\P(\K_\P)|\le cq^{d-1}$; while the Lang-Weil theorem in its
original form (using that $\G$ is absolutely irreducible) yields
$|\G(\K_\P)|\ge c'q^d$; here $c,c'$ are positive constants depending
only on $d$ and $M$. Thus $c'q^d\le k_n c q^{d-1}$, hence $k_n\ge
\frac{c'}{c}q\ge \frac{c'}{c}e^{3 \alpha n}$ and this ends the proof of the theorem modulo the claim.

Let us verify the claim. If ${f_i=0}$ is all of $\G$ modulo $\P$, this
means that $f_i$ is nilpotent in $A[\G]\otimes_{A}A/\P$. Since the latter
is a reduced ring, this means that $f_i$ is zero in
$A[\G]\otimes_{A}A/\P=A[\G]/\P A[\G]$, i.e.\ that $f_i\in \P A[\G]$. But
$A$ is a principal ideal ring, so we can write $\P=pA$, so $f_i\in pA[\G]$. By our choice of $f_i$, this implies that $p$ is invertible in $A$, a contradiction.
\end{proof}

We made use of the following classical lemma. Since we did not find a reference, we include a proof.

\begin{lemma}\label{principal}
There exists a finitely generated principal ideal subring $A$ of $\K$ containing
$\mathcal{O}_\mathbb{K}(S)$.
\end{lemma}
\begin{proof}

Recall that if $B$ is a domain with field of fractions $K$, a
fractional ideal of $B$ is by definition a nonzero finitely
generated $B$-submodule of $K$. Under multiplication, they form a
commutative semigroup with unity; if this is actually a group, $B$
is called a Dedekind domain and the quotient of this group by its
subgroup consisting of nonzero principal ideals is called the class
group of $B$ and is denoted by $\textnormal{Cl}(B)$.

Observe that if $B$ is a Dedekind domain and $D$ any multiplicative
subset of $B-\{0\}$, $D^{-1}B$ is a Dedekind domain and the natural
homomorphism $\textnormal{Cl}(B)\to\textnormal{Cl}(D^{-1}B)$ is
surjective. Moreover, if $I$ is a (finitely generated) ideal of $B$
and $D\cap I\neq \varnothing$ then the image $D^{-1}I$ of $I$ in
$\textnormal{Cl}(D^{-1}B)$ is trivial.

Now assume that $B=\mathcal{O}_\mathbb{K}(S)$, so $K=\mathbb{K}$.
Then $B$ is a Dedekind domain and $\textnormal{Cl}(B)$ is finitely
generated (it is finite in characteristic zero
\cite[Theorem~I.6.3]{neukirch} and finite-by-cyclic in positive
characteristic \cite[Lemma~5.6]{rosen}). Pick ideals $I_1,\dots,I_k$
of $B$ which are representatives of generators of
$\textnormal{Cl}(B)$, and let $s_j \in I_j \setminus \{0\}$ for each $j=1,...,k$ and $s=s_1\cdot ... \cdot s_k$. Then it follows from the remarks above that the image of each
$I_j$ in $\textnormal{Cl}(B[1/s])$ is trivial and since
$\textnormal{Cl}(B)\to\textnormal{Cl}(B[1/s])$ is surjective, we
deduce that $\textnormal{Cl}(B[1/s])$ is the trivial group,
i.e.\ $A=B[1/s]$ is a principal ideal domain.
\end{proof}

\section{Proof of uniform exponential conjugacy growth} \label{reduc}

In this section, we prove Theorem \ref{mainthm}, relying on Theorem \ref{cover}.
First, we show that without loss of generality, we may assume that the field of definition $F$ is a
global field (specialization step).  Then we reduce to the reductive case and finally prove the theorem by intersecting the ball with the semisimple part using Theorem \ref{cover} to count conjugacy classes inside the semisimple part.

\bigskip

\noindent \emph{Specialization step.}~

In proving Theorem \ref{mainthm}, the first step is to reduce the proof to the case when the field $F$ is a global field $\K$. Since $\Sigma$ is a finite set, the ring generated by the matrix entries of the elements of $\Sigma$ is a finitely generated commutative ring $R$. Such rings have lots of homomorphisms to global fields $\K$. The proposition below says that we can choose such a ring homomorphism with the property that the image of $\langle \Sigma \rangle$ under the induced homomorphism on $\langle \Sigma \rangle$ into $\GL_d(\K)$ remains non-virtually solvable. This process is traditionally called \emph{specialization}, because the ring homomorphism from $R$ to $\K$ is defined by specializing the values of a transcendence basis for $R$ to algebraic values.

\begin{proposition}[Specialization]\label{spec} Let $F$ be any field and $R$ be a finitely generated subring of $F$. Let $\Sigma$ be a finite symmetric subset of $\GL_d(R)$, which generates a non-virtually solvable subgroup $\langle \Sigma \rangle$. Then there exists a global field $\K$, with $\textnormal{char}(\K)=\textnormal{char}(F)$ and a ring homomorphism $\varphi: R \rightarrow \K$ inducing a group homomorphism $\overline{\varphi}:\langle \Sigma \rangle \rightarrow \GL_d(\K)$ such that $\overline{\varphi}(\langle \Sigma \rangle)$ is non-virtually solvable.
\end{proposition}

\begin{proof} This is now classical. See for example \cite[Proposition 2.2]{lubotzky-mann}, \cite[Theorem 4]{larsen-lubotzky} and also \cite[\S 4]{EMO} or \cite[Lemma 3.1]{breuillard-gelander}.
\end{proof}

This proposition allows us to assume that the field $F$ is
a global field $\K$ in the proof of Theorem \ref{mainthm}, because
if $g \in \langle \Sigma \rangle $, then the characteristic
polynomial $\chi_{\overline{\varphi}(g)}$ coincides with
$\varphi(\chi_g)$, so there are at least as many distinct
characteristic polynomials arising from elements in $\Sigma^n$ as
there are from elements in $\overline{\varphi}(\Sigma)^n$.

\bigskip

\noindent \emph{Reduction to a reductive group.}\newline
~
\indent Let $\G$ be the Zariski-closure of $\langle \Sigma \rangle$ in
$\GL_d(\K)$. Recall, by definition, that a reductive algebraic group
is an algebraic group with no non-trivial unipotent normal subgroup.
Note that we do not require reductive groups to be connected here.
We have:

\begin{lemma}[Going to the reductive part]\label{reductivepart} Let $\G \subset \GL_d$ be an algebraic group defined over a field $\K$. There is a homomorphism of algebraic groups $\rho: \G \rightarrow \GL_d$ defined over $\K$, and with unipotent kernel, such that $\rho(\G)$ is a reductive algebraic subgroup of $\GL_d$ defined over $\K$ and such that $\chi(\rho(g))=\chi(g)$, for every $g \in \G(\K)$, where $\chi(g)$ is the characteristic polynomial of $g$ in $\GL_d$.
\end{lemma}

\begin{proof}
Let $V=\K^d$ and let $\mathbf{U}$ be the maximal normal unipotent
subgroup of $\G$. Being unipotent, $\mathbf{U}$ admits a non-trivial
subspace of fixed points $V_1$ in $V$, and in fact stabilizes a flag
$V_1 \subsetneq V_2 \subsetneq ... \subsetneq V_r=V$, such that
$V_i/V_{i-1}$ consists of the $\mathbf{U}$-fixed points in $V/V_{i-1}$.
Then $\G$ leaves each $V_i$ invariant and its action on
$V_i/V_{i-1}$ factors through $\G/\mathbf{U}$. Replace the original
representation by its semi-simplification, i.e.\ the representation
$\rho$ on $V=\oplus_i V_i/V_{i-1}$. It is easy to see that the new
representation consists of the diagonal blocks of the old one and gives rise to
the same characteristic polynomial as the old one, and that the kernel of $\rho$ is unipotent.
\end{proof}

Accordingly, to prove Theorem \ref{mainthm}, it is enough to do it under the additional assumptions that the field $F$ is a global field, and the Zariski closure of $\Sigma$ is reductive (possibly not connected): indeed applying Lemma \ref{reductivepart} to the Zariski closure of the subgroup generated by $\Sigma$, since the kernel of $\rho$ is nilpotent, the image of $\Sigma$ still generates a non-virtually-solvable subgroup.
What we actually show is the following. Recall that $\alpha_\Sigma$ was defined in $(\ref{defalpha})$.

\begin{proposition}\label{explimain}
For every $d$, there exists a constant $\eta(d)>0$ such that if $\K$ is a global field and $\Sigma$ a finite symmetric subset of $\GL_d(\K)$ generating a non-virtually-solvable subgroup with (not necessarily connected) reductive Zariski closure $\G$, then $$\liminf_{n \rightarrow \infty} \frac{1}{n} \log \chi_{\Sigma}(n)\geq \eta(d)\alpha_\Sigma.$$
\end{proposition}

According to the uniform exponential growth of linear groups \cite{breuillard1}, $\alpha_\Sigma$ can be bounded below by a positive constant $c(d)$, not depending on $\K$ nor on the subgroup generated by $\Sigma$. Therefore, in view of the reductions above, Theorem \ref{mainthm} follows from Proposition \ref{explimain}, which we now proceed to prove.

\begin{proof}[Proof of Proposition \ref{explimain}.]
Let $\G^0$ be the connected component
of the identity in $\G$, and let $\mathbf{H}:=[\G^0,\G^0]$ be the
commutator subgroup of $\G^0$. Then $\mathbf{H}$ is
a connected semisimple algebraic group. Let $\pi: \G \rightarrow \G/\mathbf{H}$ be the quotient homomorphism; we have $\ker \pi=\mathbf{H}$. Let also $\mathbf{S}_i$ be the absolutely simple factors of $\mathbf{H}$ and $\pi_i: \G^0 \rightarrow \mathbf{S}_i$ the canonical projections. Up to passing to a finite extension of $\K$ if necessary, we may assume that the $\mathbf{S}_i$ and the projection maps $\pi_i$ are defined over $\K$.

Since $\G/\mathbf{H}$ is virtually abelian, the growth of $|\pi(\Sigma^n)|$ is at most polynomial, say $\le Cn^\kappa$. By the pigeonhole principle, there must exist a coset of $\mathbf{H}$ whose intersection with $\Sigma^n$ has as least $|\Sigma^n|/Cn^\kappa$ elements. It follows that $|\Sigma^{2n} \cap \mathbf{H}| \ge |\Sigma^n|/Cn^\kappa$. Moreover, setting $\alpha_\Sigma(i):= \liminf_{n \rightarrow \infty} \frac{1}{n} \log |\pi_i(\Sigma^{2n} \cap \mathbf{H})|$, we have,
$$d \max_i \alpha_\Sigma(i) \geq \sum_i \alpha_\Sigma(i) \geq \liminf_{n \rightarrow \infty} \frac{1}{n} \log |\Sigma^{2n} \cap \mathbf{H}| \geq \alpha_\Sigma.$$

Let $j$ be an index  such that $\alpha_\Sigma(j)=\max_i \alpha_\Sigma(i)$.
Let $B_n=\pi_j(\Sigma^{2n} \cap \mathbf{H})$. We have: $$\liminf_{n \rightarrow \infty} \frac{1}{n} \log |B_n| \geq \frac{1}{d}\alpha_\Sigma$$We are going to apply Theorem \ref{cover} to the simple group $\mathbf{S}_j$, the $B_n$'s and the subvarieties of $\mathbf{S}_j$ defined by $\mathcal{V}_f:=\overline{\pi_j(\{g \in \mathbf{H}, \chi_g = f\})}$, where $f\in \K[X]$ is an arbitrary polynomial and $\chi_g$ denotes the characteristic polynomial of $g$. Let $\alpha< \frac{1}{d}\alpha_\Sigma$. We now check that the assumptions of that theorem do hold.

The $\mathcal{V}_f$ are subvarieties of $\mathbf{S}_j$ whose complexity is bounded in terms of $d$ only and in particular independently of $f$. Let us check that they are proper subvarieties too. Let $T$ be a maximal torus of $\mathbf{H}$ and $\lambda_i$'s be characters of $T$ in the ambient linear representation of $\mathbf{H}$, so that for every $t \in T$, $\chi_t(X)=\prod_i (\lambda_i(t) - X)$. Write $T=T_1T_2$, where $T_1 \cap T_2$ is finite and $T_1$ is isogenous via $\pi_j$ to a maximal torus of $\mathbf{S}_j$. If $\mathcal{V}_f$ were not proper, then for a dense set of $t_1 \in T_1$, there would exist a $t_2 \in T_2$ such that $\chi_{t_1t_2}(X)=f=\prod_i(\lambda_i - X)$. We would thus have  $\lambda_i(t_1t_2)=\lambda_i$ for all $i$. But recall that if $t \in T$, then $\lambda_i(t)=1$ for all $i$ implies $t=1$. Since $T_1 \cap T_2$ is finite, this implies that $T_1$ is finite, which is impossible. We conclude that the $\mathcal{V}_f$'s are proper subvarieties of $\mathbf{S}_j$.

The $B_n$'s form an increasing family of symmetric subsets of $\mathbf{S}_j$ with $|B_n| \geq e^{\alpha n}$ for all $n$ large enough. Moreover, observe that $\Lambda:=\langle \Sigma \rangle \cap \G^0$ is finitely generated since $\G^0$ has finite index in $\G$. It follows from the Reidemeister-Schreier rewriting process (see \cite[sec 2.3]{magnus}) that there exists a finite set of generators $W_0$ of $\Lambda$ such that for every $\gamma \in \Lambda$ one has $\ell_{W_0}(\gamma) \leq \ell_{\Sigma}(\gamma)$, where $\ell_{W_0}$ and $\ell_{\Sigma}$ denote the word length with respect to the generating sets $W_0$ and $\Sigma$. Taking $W:=W_0W_0$, we get a finite set $W \subset \langle \Sigma \rangle \cap \G^0$ such that $\Sigma^{2n} \cap \G^0 \subset W^{n}$, and hence $B_n \subset \pi_j(W)^n$. It now only remains to check that $\langle B_n \rangle$ eventually contains some fixed Zariski-dense subgroup $\Gamma$ of $\mathbf{S}_j$. We require the following lemma:

\begin{lemma} Let $\mathbf{H}$ be a connected semisimple algebraic group defined over a global field $\K$ and $\Delta$ be a Zariski-dense subgroup of $\mathbf{H}(\K)$. Then $\Delta$ contains a finitely generated Zariski-dense subgroup $\Gamma$.
\end{lemma}

\begin{proof} The argument is standard. For each simple factor $\mathbf{S_i}$ of $\mathbf{H}$, one can find an element $\sigma_i$ in $\Delta$ whose projection to $\mathbf{S_i}$ has infinite order (note that $\K$ has only finitely many roots of unity). If the connected component $\mathbf{L}$ of the Zariski closure of the subgroup generated by the $\sigma_i$ is normal in $\mathbf{H}$ we are done, because it maps nontrivially on all $\mathbf{S_i}$'s. If not, then one can find $\gamma_j \in \Delta$ such that the Zariski closure of $\langle \mathbf{L}, \gamma_j \mathbf{L} \gamma_j^{-1}\rangle$ has dimension $> \dim \mathbf{L}$. This process must stop after at most $\dim \mathbf{H}$ steps, and the $\sigma_i$'s together with the $\gamma_j \sigma_i \gamma_j^{-1}$'s generate a Zariski dense subgroup of $\mathbf{H}$.
\end{proof}

Note that $\langle \Sigma \rangle \cap \mathbf{H}$ is Zariski-dense in $\mathbf{H}$ because $\langle \Sigma \rangle \cap \mathbf{G}^0$
 is Zariski-dense in $\G^0$ and the commutator map is surjective from $\G^0 \times \G^0$ to $\mathbf{H}$. Thus the lemma applied to $\Delta:=\langle \Sigma \rangle \cap \mathbf{H}$ implies that $\langle \Sigma \rangle \cap \mathbf{H}$ contains a finitely
 generated Zariski dense subgroup in $\mathbf{H}$, and hence $\pi_j(\langle \Sigma \rangle \cap \mathbf{H})$
  contains a finitely generated subgroup $\Gamma$ which is Zariski dense in $\mathbf{S}_j$. Hence $\langle B_n \rangle$ will eventually contain $\Gamma$.
  We have now checked that the assumptions of Theorem \ref{cover} hold in our situation and we can conclude that
   $$\chi_\Sigma(2Dn) \geq \Theta_n \geq e^{\alpha n},$$ as soon as $n$ is large enough. This implies
$$\liminf_{n \rightarrow \infty} \frac{1}{n} \log \chi_{\Sigma}(n)\geq \frac{\alpha}{2D};$$
since this holds whenever $\alpha<\frac{1}{d}\alpha_\Sigma$. This completes the proof of Proposition \ref{explimain}.
\end{proof}

\begin{remark}It would have been more elegant to reduce to the semisimple case by finding a subset $\Sigma' \subset \Sigma^N \cap \mathbf{H}$ such that $\Sigma'$ generates a Zariski-dense subgroup of $\mathbf{H}$. Unless the characteristic is zero, we cannot afford doing this here, because $\G/\G^0$ and hence $N$ cannot be uniformly bounded in terms of $d$ only and proceeding in this way would ruin the uniformity in Theorem \ref{mainthm}.
\end{remark}

\section{Concluding remarks and suggestions for further research}\label{concluding}

\subsection*{Images of balls under regular maps} In this subsection, we give some further applications of the method of this paper. Using Theorem \ref{cover} and working directly with subvarieties of the simple group $\G$, the proof of Theorem \ref{mainthm} generalizes straightforwardly to yield:

\begin{theorem}\label{fiber}
Let $d \geq 1$. There exists a constant $c=c(d)>0$ such that the following holds. Let $F$ be a field, $\G$ a $d$-dimensional absolutely simple algebraic group defined over $F$ and $\Gamma$ a Zariski-dense subgroup of $\G$ generated by a finite set $\Sigma$. Let $f$ be a regular
function on $\G$ defined over $F$. Assume that $f$ is nonconstant on $\G$. Then the
image of the $\Sigma^n$ under $f$ grows at an exponential rate at
least $c$, i.e.
$$\liminf_{n \to \infty}\frac{1}{n}\log |f(\Sigma^n)| \ge c.$$
\end{theorem}

\begin{proof}
Applying \cite[Theorem 4]{larsen-lubotzky}, we may specialize as in Proposition \ref{spec} to a global field $\K$ with the additional property that the image of $\Gamma$ under the specialization map is still Zariski-dense in $\G$. Then the conditions of Theorem \ref{cover} are fulfilled with $B_n=\Sigma^{n}$, $W=\Sigma$, the $\mathcal{V}_i$ being level sets of the regular map $f$, and $\alpha>0$ gotten from uniform exponential growth \cite{breuillard1}. Setting $c=\alpha/D$, where $D$ is the constant obtained in Theorem \ref{cover}, we are done.
\end{proof}

This can be applied for example to the trace function:

\begin{corollary}\label{trace-cor} Assume $\G \leq \GL_d$ is a connected simple algebraic group over a field $F$ on which the restriction of the trace function $g \mapsto \textnormal{Trace}(g)$ is not constant. Then for every finite $\Sigma \subset \G(F)$ generating a Zariski dense subgroup in $\G$, we have

$$\liminf_{n \to \infty}\frac{1}{n}\log |\{\textnormal{Trace}(g); g \in \Sigma^n\}| \ge c, $$
for some constant $c>0$ depending only on $d$ (and not on $F$ nor $\Sigma$).
\end{corollary}

It can happen that the trace function is constant on some simple groups, e.g.\ if the characteristic is $p$ and $\G=\SL_n$ is embedded diagonally in $\SL_{np}$. But one can show that if $\G$ is any Zariski connected algebraic subgroup of $\GL_d$, which is not unipotent, then if the characteristic of $F$ is either $0$ or finite and more than $d$, then the trace function is not constant on $\G$.

\subsection*{Solvable groups}\label{solvable_section} In \cite{breuillard-cornulier} it was proved that virtually solvable groups have exponential conjugacy growth unless they are virtually nilpotent. By way of contrast, this does not hold when we look at $\GL_d$-conjugacy classes.

\begin{proposition}\label{growsol}
Let $\Sigma$ be a finite subset of $\GL_d$ over any field, generating a virtually solvable group $\Gamma$. Then the number of characteristic polynomials $\chi_\Sigma(n)$ is polynomially bounded. Moreover, it is bounded if and only if $\Gamma$ is virtually unipotent.
\end{proposition}
\begin{proof}
It will be convenient to prove the following equivalent statement. Let $\Gamma$ be a group with a finite generating subset $\Sigma$ and let $\rho:\Gamma\to\GL_d$ be a linear representation over any field with virtually solvable image. Let $\chi_\Sigma^\rho(n)$ be the number of distinct characteristic polynomials in $\rho(B_\Sigma(n))$. Then $\chi_\Sigma^\rho(n)$ is polynomially bounded with respect to $n$. Moreover, it is
bounded if and only if $\rho(\Gamma)$ is virtually unipotent.

Let us prove the latter statement. First, let $\pi$ be the
semisimplification of $\rho$ (see the proof of Lemma \ref{reductivepart}). Then
$\chi_\Sigma^\rho=\chi_\Sigma^\pi$ and $\pi$ has virtually solvable
image (since $\ker\pi \supset \ker\rho$). Now let $G$ be the Zariski
closure of $\pi(\Gamma)$; since its action is semisimple, $G$ is
reductive, and since $G$ is virtually solvable, it is therefore
virtually abelian. So $\pi(\Gamma)$ is virtually abelian and hence
has polynomial growth. It follows that $\chi_\Sigma^\rho$ is
polynomially bounded.

For the last statement of the proposition, observe that if $\rho(\Gamma)$ is virtually unipotent then $\pi(G)$ has finite image; conversely if $\rho(\Gamma)$ is not virtually unipotent, then $\Gamma$ contains some element with an eigenvalue which is not a root of unity, hence $\chi_\Sigma^\rho$ is unbounded.
\end{proof}

Since in $\GL_d$ there are at most $O_d(1)$ conjugacy classes with a given characteristic polynomial, we deduce
\begin{corollary}
If $\Gamma$ is a virtually solvable subgroup of $\GL_d$ then the
number of $\GL_d$-conjugacy classes met by the $n$-ball in $\Gamma$
is polynomially bounded.
\end{corollary}

New questions arise if we ask about the number of $\G$-conjugacy classes in $\Sigma^n$, especially when $\G$ is the Zariski closure of $\Gamma$. Let us provide two examples where different phenomena appear.

\begin{example}
Let $\G$ be the group of upper triangular $3\times 3$ matrices
$(a_{ij})$ with $a_{11}=a_{33}=1$. Set $\Gamma=\G(\Z[1/2])$. Then
the reader can check that $\Gamma$ is finitely generated and Zariski
dense in $\G$. Moreover, its conjugacy growth is exponential, as the
elements $\begin{pmatrix} 1 & 0 & k\\ 0 & 1 & 0\\ 0 & 0 &
1\end{pmatrix}$, for $k=0,1,\dots,2^n$ have word length $O(n)$ but
are pairwise non-conjugate in $\G$.
\end{example}

\begin{example}
We present an example where the type of conjugacy growth depends on the field. Let $\G$ be either $\SL_2$ or its subgroup consisting of upper triangular matrices. Let $\Gamma$ be the subgroup generated by $\begin{pmatrix}2 & 0\\ 0 & 1/2\end{pmatrix}$ and $\begin{pmatrix}1 & 1\\ 0 & 1\end{pmatrix}$. Then the elements $\begin{pmatrix}1 & p\\ 0 & 1\end{pmatrix}$ for $p$ prime in $[0,2^n]$ (there are exponentially many such elements) have word length $O(n)$ and are pairwise non-conjugate in $\G(\mathbb{Q})$. On the other hand, every element in the $n$-ball is conjugate in $\G(\C)$ to $\begin{pmatrix}1 & 1\\ 0 & 1\end{pmatrix}$ or $\begin{pmatrix}2^k & 0\\ 0 & 2^{-k}\end{pmatrix}$ for some $k$ with $-n\le k\le n$. So $\Gamma$ has exponential $\G(\Q)$-conjugacy growth but linear $\G(\C)$-conjugacy growth.
\end{example}

\subsection*{On the rate of exponential growth} We record here a related open problem.
Let $\Gamma$ be a group and $\Sigma$ a symmetric generating subset. In general, we have
$$\gamma_\Sigma:=\liminf_{n \rightarrow \infty} \frac{\log c_\Sigma(n)}{n}\le \limsup_{n \rightarrow \infty} \frac{\log c_\Sigma(n)}{n}\le \lim_{n \rightarrow \infty} \frac{\log|B_\Sigma(n)|}{n}=\alpha_\Sigma.$$
As we saw in the introduction, Osin's groups provide examples for which the inequality on the right-hand side is strict. We are not aware of any example for which inequality on the left-hand side is strict but constructions of the same spirit might provide examples. On the other hand, for non-virtually-solvable linear groups, does $\gamma_\Sigma = \alpha_\Sigma$ hold? in fact we do not know if any of those two inequalities can be sharp. For instance, in a free group over $\Sigma$, it is easy to check that both are equalities.

In case $\K$ is a global field and $\Gamma$ is a non-virtually-solvable subgroup of $\GL_d(\K)$ whose Zariski closure is reductive, Proposition \ref{explimain} implies that $\gamma_\Sigma\ge\eta(d)\alpha_\Sigma$, where $\eta(d)>0$ only depends on $d$. It would be interesting to investigate if these assumptions (i.e.\ $\K$ be a global field, the Zariski closure be reductive) could be relaxed.

\setcounter{tocdepth}{1}

\end{document}